\documentclass[a4paper,12pt,leqno]{amsart}
\usepackage{amsmath,amssymb,mathrsfs}
\usepackage[matrix,arrow]{xy} 
\usepackage{xcolor,hyperref}
\hypersetup{colorlinks=true,citecolor=black,linkcolor=black}
\newcommand{\h}{\hbox}
\newcommand{\q}{\quad}
\newcommand{\nin}{\noindent}

\newcommand{\ms}{\par\medskip}
\newcommand{\sk}{\par\smallskip}

\newcommand{\msn}{\par\medskip\noindent}

\newcommand{\ges}{\geqslant}
\newcommand{\les}{\leqslant}
\newcommand{\one}{\hskip1pt}
\newcommand{\mcap}{\hbox{$\bigcap$}}
\newcommand{\mcup}{\hbox{$\bigcup$}}
\newcommand{\msum}{\hbox{$\sum$}}
\newcommand{\mopl}{\hbox{$\bigoplus$}}
\newcommand{\mprod}{\hbox{$\prod$}}
\newcommand{\B}{{\mathcal B}}
\newcommand{\D}{{\mathcal D}}

\newcommand{\OO}{{\mathcal O}}

\newcommand{\PP}{{\mathbb P}}
\newcommand{\Q}{{\mathbb Q}}
\newcommand{\C}{{\mathbb C}}
\newcommand{\N}{{\mathbb N}}
\newcommand{\R}{{\mathbb R}}

\newcommand{\Z}{{\mathbb Z}}

\newcommand{\Dt}{\widetilde{D}}

\newcommand{\Yt}{\widetilde{Y}}

\newcommand{\Gp}{\Gamma_{\!+}}
\newcommand{\al}{\alpha}
\newcommand{\be}{\beta}
\newcommand{\ga}{\gamma}
\newcommand{\Ga}{\Gamma}
\newcommand{\de}{\delta}

\newcommand{\la}{\lambda}

\newcommand{\si}{\sigma}

\newcommand{\om}{\omega}

\newcommand{\dd}{\partial}
\newcommand{\ddd}{{\rm d}}
\newcommand{\Gr}{{\rm Gr}}

\newcommand{\tos}{\,{\to}\,}
\newcommand{\eq}{\,{=}\,}
\newcommand{\defs}{\,{:=}\,}
\newcommand{\nes}{\,{\ne}\,}
\newcommand{\ins}{\,{\in}\,}

\newcommand{\sst}{\,{\subset}\,}
\newcommand{\stm}{\,{\setminus}\,}
\newcommand{\gess}{\,{\ges}\,}
\newcommand{\less}{\,{\les}\,}
\newcommand{\sgt}{\,{>}\,}
\newcommand{\slt}{\,{<}\,}
\newcommand{\col}{\,{:}\,}
\newcommand{\pl}{\one {+}\one}
\newcommand{\mi}{\one {-}\one}
\newcommand{\bl}{\bigl}
\newcommand{\br}{\bigr}

\newcommand{\ssb}{\raise.15ex\h{${\scriptscriptstyle\bullet}$}}
\newcommand{\ssc}{\,\raise.15ex\h{${\scriptstyle\circ}$}\,}

\newcommand{\simto}{\,\,\rlap{\hskip1.5mm\raise1.4mm\hbox{$\sim$}}\hbox{$\longrightarrow$}\,\,}

\makeatletter
\renewcommand\section{\@startsection{section}{1}{0pt}{-3ex plus -1ex minus -.2ex}{2.3ex plus.2ex}{\centering\normalfont\bfseries}}
\makeatother
\theoremstyle{plain}
\newtheorem{thm}{Theorem}[section]

\newtheorem{prop}[thm]{Proposition}
\newtheorem{lem}[thm]{Lemma}
\newtheorem{ithm}{Theorem}
\newtheorem{icor}{Corollary}

\newtheorem{ilem}{Lemma}

\theoremstyle{definition}
\newtheorem{rem}[thm]{Remark}

\newtheorem{irem}{Remark}

\begin{document}
\title[Strong monodromy conjecture]{Strong monodromy conjecture for defining polynomials of projective hypersurfaces having\\only weighted homogeneous isolated singularities}
\author[M. Saito]{Morihiko Saito}
\address{M. Saito : RIMS Kyoto University, Kyoto 606-8502 Japan}
\begin{abstract} Let $Z\subset{\bf P}^{n-1}$ be a hypersurface such that the associated reduced hypersurface $Z_{\rm red}$ has only weighted homogeneous isolated singularities. In the case $Z$ is a reduced curve or $Z_{\rm red}$ has only homogeneous isolated singularities with $n$ at least $4$, we show that the strong monodromy conjecture for a defining polynomial $f$ of $Z$ follows from arxiv:1609.04801v1 using in the reduced curve case a formula of Denef and Loeser for Newton-nondegenerate polynomials of three variables (which can be deduced in the applied case from the one for the two variable case) together with known results about the strong monodromy conjecture in the two variable case. Here an amazing cancellation occurs so that possible counterexamples fail. We also show the relation between the pole orders of topological zeta function and the root multiplicities of Bernstein-Sato polynomial in the case $Z$ has equimultiplicity and $Z_{\rm red}$ has only weighted homogeneous singularities with $n=3$ or $Z_{\rm red}$ has only homogeneous isolated singularities with $n>3$.
\end{abstract}
\maketitle

\section*{Introduction}
\nin
Let $f$ be a defining polynomial of a projective hypersurface $Z\sst\PP^{n-1}$ such that the associated reduced hypersurface $Z_{\rm red}$ has only weighted homogeneous singularities. Set $d\defs\deg Z$. We have the following (which is an immediate consequence of \cite[Theorem 2]{wh1} and \cite[Theorem 2]{bcm}, see also \cite[Theorem C]{BV} including the non-reduced case).

\begin{ithm} \label{T1}
Assume $Z$ is reduced and there is no nonzero vector field of degree $0$ which annihilates $f$ and whose trace is nonzero. Then $-\tfrac{n}{d}$ is a root of the Bernstein-Sato polynomial $b_f(s)$ of $f$.
\end{ithm}

Here the trace of a vector field of degree $0$ means the trace of the matrix consisting of the coefficients $c_{i,j}$ of a vector field $\msum_{i,j}\,c_{i,j}x_i\dd_{x_j}$ of degree 0 with $c_{i,j}\ins\C$. This theorem immediately follows from the $E_2$-degeneration of the pole order spectral sequence (see \cite[Theorem 2]{wh1} using \cite[Theorem 2]{bcm}, since the image of the $E_1$-differential $\ddd_1$ is given by the traces of annihilating vector fields of degree 0, see Remark~\ref{R4.3} below.
\sk
Employing only Jordan normal form, we can verify in a completely elementary way the following (which is a natural generalization of an argument in the proof of \cite[Proposition 8.1]{wh}, see also \cite[Lemma 3.11]{BV}).

\begin{ilem} \label{L1}
Assume $Z$ is reduced. If $f$ is annihilated by a nonzero vector field $\xi$ of degree~$0$, then $f$ is also annihilated by the semisimple part $\xi'$ of $\xi$ $($which is defined by using Jordan normal form$)$.
\end{ilem}

The trace of a vector field of degree 0 is zero if its semisimple part vanishes. So Lemma~\ref{L1} implies the following (which can replace \cite[Propositions 8.1-2]{wh}).

\begin{icor} \label{C1}
Assume $Z$ is reduced. If the second hypothesis of Theorem~{\rm\ref{T1}} is not satisfied, then $f$ is extremely degenerated, that is, it is annihilated by a nonzero ``diagonal'' vector field of degree $0$ of the form $\msum_{i=0}^n\,c_ix_i\dd_i$ with $c_i\ins\Q$.
\end{icor}

Here we have $c_i\ins\Q$ by using a $\Q$-basis $\ga_j$ of $\msum_{i=0}^n\,\Q\,c_i\sst\C$ so that $c_i\eq\msum_j\,\ga_jc'_{j,i}$, where we may get several annihilating vector fields of degree 0 with rational coefficients. (Note that monomials are $\C$-{\it linearly independent\one} in the polynomial ring. This is the key point in the proof of Lemma~\ref{L1}.)
\sk
In the reduced curve case, using Theorem~\ref{T1} and Corollary~\ref{C1}, it is easy to reduce the {\it strong monodromy conjecture\one} for the {\it local topological zeta function\one} $Z^{\rm top}_{f,0}(s)$ (claiming that a pole of $Z^{\rm top}_{f,0}(s)$ is a root of the Bernstein-Sato polynomial $b_f(s)$, see \cite{DL}, \cite{BSY}, \cite{bha}, \cite{Wa}, \cite{BBV}, \cite{BV} among others) to the following  (using also some known results in the two variable case \cite{Lo}, see Remark~\ref{R3.1} below).

\begin{ithm} \label{T2}
If $Z$ is a reduced curve $C$ and is extremely degenerated, then any pole of the local topological zeta function $Z^{\rm top}_{f,0}(s)$ is a pole of the local topological zeta function of a singular point of $C$.
\end{ithm}

This can be verified rather easily if one admits a formula for the local topological zeta function of a Newton-nondegenerate polynomial of three variables, see \cite{DL} (which is quite amazing). It is however not completely trivial, since a ``fake" pole may appear associated with a $B_1$-facet in the sense of \cite{LV}, see the proof of Theorem~\ref{T2} in Section~\ref{S3} below.
In the applied case it is possible to deduce the formula from the one for the two variable case (which seems much easier to prove) using the blowup at the origin followed by the ones with line centers together with some computations of rational functions by a computer algebra program such as Macaulay2 or Singular among others. Here it is very surprising that a pole $\tfrac{1}{ds+3}$ {\it disappears\one} (more precisely, the numerator is divisible by $ds\pl3$), when one calculates $Z^{\rm top}_{f,0}(s)$ using a desingularization even if the Euler number of $\PP^2\stm C$ does not vanish so that one could expect a counterexample, see Remark~\ref{R3.2} below. Note also that a solution of the strong monodromy conjecture for the local topological zeta function in a reduced case implies that for a non-reduced case if the multiplicities are constant, see Remark~\ref{R1} below.
\sk
The strong monodromy conjecture has been shown quite recently in the not necessarily reduced curve case under the hypothesis that the associated reduced curve has only weighted homogeneous singularities (see \cite{BV}), where the vanishing of the residue at $-\tfrac{3}{d}$ is proved in the needed case. Here it does not seem to be mentioned that the following assertion actually holds.

\begin{ithm} \label{T3}
Let $Z\sst\PP^{n-1}$ be a projective hypersurface defined by a polynomial $f$ with $n\gess4$. Assume the reduced hypersurface $Z_{\rm red}$ has only homogeneous isolated singularities. Then the strong monodromy conjecture holds for the topological zeta function of $f$.
\end{ithm}

This is reduced to Theorem~\ref{T1} and Corollary~\ref{C1} using Proposition~\ref{P4.1} and Remark~\ref{R1} below, since the strong monodromy conjecture for a homogeneous polynomial with an isolated singularity is easy. Here that $Z$ is {\it irreducible\one} since $n\gess4$. (Note that the local topological zeta function coincides with the global one in the case of a homogeneous polynomial, and this is the same for the Bernstein-Sato polynomial.)

\begin{irem}\label{R1}
For a positive integer $m$, the poles of the local topological zeta function for $f^m$ are those for $f$ divided by $m$. Indeed, the $N_i$ are multiplied by $m$ (although the $\nu_i$ do not change) by replacing $f$ with $f^m$ in the notation of Section~\ref{S1} below. Moreover the roots of the Bernstein-Sato polynomial $b_f(s)$ divided by $m$ are contained in the roots of $b_{f^m}(s)$, see the last divisibility in Remark~\ref{R4.2} below. So the strong monodromy conjecture for the local topological zeta function of $f^m$ is reduced to that for $f$.
\end{irem}

As for the relation between the orders of poles of $Z^{\rm top}_{f,0}(s)$ and the multiplicities of roots of $b_f(s)$, we can show the following.

\begin{ithm}\label{T4}
If $n\eq3$, assume the curve $Z\sst\PP^{n-1}$ has equimultiplicity and the reduced curve $Z_{\rm red}$ has only weighted homogeneous singularities. In the case $n\sgt3$, assume the reduced $Z_{\rm red}$ has only homogeneous isolated singularities. Then for any pole $\al$ of the topological zeta function $Z^{\rm top}_{f,0}(s)$, its multiplicity is at most that of $\al$ as a root of the Bernstein-Sato polynomial $b_f(s)$, that is, the product $b_f(s)\one Z^{\rm top}_{f,0}(s)$ has no poles as a rational function of $s$.
\end{ithm}

This follows from (an extension of) \cite[Theorem 3.8]{bha} using \cite[Propositions 8.1--2]{wh} (see also \cite[p.~123]{dPW}) together with \cite{We} and \cite[Theorem 2.4 and Lemma 6.6]{dPW2}. Note that for a homogeneous polynomial, the local topological zeta function coincides with the global one, and this is the same for the Bernstein-Sato polynomial.
\sk
We thank S.-J.\ Jung, D.\ Bath, and N.\ Budur for useful comments about this note.

\tableofcontents
\numberwithin{equation}{section}

\section{Topological zeta function} \label{S1}

Let $\pi\col\Yt\tos Y\defs\C^n$ be an embedded resolution of singularities of $D\defs f^{-1}(0)$ with $D_i$ the irreducible components of $\Dt\defs\pi^{-1}(D)$. Let $N_i$, $\nu_i\mi1$ be the multiplicities of $\pi^*f$ and $\pi^*\ddd x_1{\wedge}\cdots{\wedge}\ddd x_n$ respectively. Set $D_I\defs\mcap_{i\in I}\,D_i$, $D^{\circ}_I\defs D_I\stm\mcup_{i\notin I}R_i$. The local topological zeta function is defined by
\begin{equation*}
Z^{\rm top}_{f,0}(s):=\msum_I\,\chi\bl(D^{\circ}_I\cap\pi^{-1}(0)\br)\,\mprod_{i\in I}\,(N_is\pl\nu_i)^{-1}.
\end{equation*}
This is independent of a resolution, see \cite{DL}.

\begin{rem}\label{R1.1}
In the curve case with $n\eq2$, the Euler number $\chi\bl(D^{\circ}_I\cap\pi^{-1}(0)\br)$ vanishes if $|I|\eq1$ and $|\mcup_{j\ne i}D_i\cap D_j|\eq2$, but the summands $\mprod_{i\in I}\,(N_is\pl\nu_i)^{-1}$ for $|I|\eq2$ remain if $\mcap_{i\in I}\,D_i\ne\emptyset$ and they may cancel each other out by using
\begin{equation*}
(s\pl\al)^{-1}\mi(s\pl\be)^{-1}\eq(\be\mi\al)(s\pl\al)^{-1}(s\pl\be)^{-1}.
\end{equation*}
In the reduced curve case it turns out that the poles of $Z^{\rm top}_{f,0}(s)$ are restricted to either $-1$ or $-\nu_i/N_i$ for exceptional divisors $D_i\sst\pi^{-1}(0)$ with $|\mcup_{j\ne i}D_i\cap D_j|\gess3$, see \cite[Theorem 4.3]{Ve}.
\end{rem}

\section{Newton-nondegenerate case}\label{S2}

We denote by $\Gp(f)$ the {\it Newton polytope\one} of $f$ at 0, which is the convex hull of the union of $\nu\pl\R_{\ges 0}^n$ for $\nu\ins{\rm Supp}_{\bf x}\,f$, where
\begin{equation*}
{\rm Supp}_{\bf x}\,f:=\{\nu\ins\N^n\mid c_{\nu}\ne 0\one\}\q\h{for}\q f\eq\msum_{\nu}\,c_{\nu}{\bf x}^{\nu}\ins\C[{\bf x}],
\end{equation*}
with ${\bf x}\eq(x_1,\dots,x_n)$ the coordinate system of $\C^n$.
We say that $f$ is {\it Newton nondegenerate,} if for any {\it compact\one} face $\si\sst\Gp(f)$, we have
\begin{equation*}
\mcap_{i=1}^n\,\bl\{\one x_i\dd_{x_i}f_{\si}\eq0\br\}\sst\{\one x_1\cdots x_n\eq0\one\},
\end{equation*}
where $f_{\si}\defs\mopl_{\nu\in\si}\,c_{\nu}{\bf x}^{\nu}$ with $c_{\nu}$ as above. We denote by $\Ga_{\!f}$ the union of compact faces of $\Gp(f)$. Set $d_{\si}\defs\dim\si$ for $\si<\Gp(f)$ (that is, $\si$ is a face of $\Gp(f)$). 
We have the following.

\begin{thm}[{\cite[Theorem 5.3]{DL}}]\label{T2.1}
If $f$ is Newton-nondegenerate, there is the equality
\begin{equation}\label{2.1}
Z^{\rm top}_{f,0}(s)=\msum_{d_{\si}=0}\,J_{\si}(s)+\tfrac{s}{s+1}\,\msum_{d_{\si}>0}\,(-1)^{d_{\si}}d_{\si}!\one{\rm Vol}(\si)J_{\si}(s),
\end{equation}
where the $\si$ are compact faces of $\Gp(f)$, that is, $\si<\Ga_{\!f}$.
\end{thm}

We explain the notation in the case $n\eq3$, see \cite{DL}, \cite{LV} for the general case. (This is similar and easier when $n\eq2$.)
Firstly $d_{\si}!\one{\rm Vol}(\si)$ is the {\it normalized volume\one} of $\si$. In the case $d_{\si}\eq0$ or 1, it is equal to 1 or $|\si\cap\N^3|\mi1$ respectively. If $d_{\si}\eq2$ and $\si$ is simplicial so that the vertices of $\si$ are three points $a_1,a_2,a_3\ins\N^3$, it coincides with the absolute value of the determinant of the matrix $(a_1,a_2,a_3)$ divided by the {\it lattice-theoretic distance\one} $N_{\si}$ between $\si$ and 0. The latter is defined by $N_{\si}\defs\ell_{\si}(a_i)$ for $i\ins\{1,2,3\}$ with $\ell_{\si}$ the linear function with integral coefficients whose greatest common divisor is 1 and such that $\ell_{\si}$ is orthogonal to $\si$, that is, $\si\sst\ell_{\si}^{-1}(N_{\si})$. We put $\nu_{\si}\defs\ell_{\si}(1,1,1)$, which is equal to the sum of the coefficients of $\ell_{\si}$. In general the normalized volume is defined by taking a partition of $\si$. Note that $\ell_{\si}$, $N_{\si}$, $\nu_{\si}$ are well defined for non-simplicial 2-dimensional $\si$. These can be defined also for non-compact faces (after taking the projection to $\R^2$ or $\R$ if necessary), where $N_{\si}\eq0$ or $1$ if two of the coefficients of $\ell_{\si}$ vanishes (since $f$ is reduced).
\sk
If $d_{\si}\eq2$, put
\begin{equation*}
J_{\si}(s):=1/(N_{\si}s\pl\nu_{\si}).
\end{equation*}
\sk
If $d_{\si}\eq1$ and $\si\eq\si_1\cap\si_2$ with $d_{\si_i}\eq2$ for $i\eq1,2$, set
\begin{equation*}
J_{\si}(s):={\rm mult}(\ell_{\si_1},\ell_{\si_2})\big/\bl((N_{\si_1}s\pl\nu_{\si_1})(N_{\si_2}s\pl\nu_{\si_2})\br).
\end{equation*}
Here ${\rm mult}(\ell_{\si_1},\ell_{\si_2})$ is defined by the greatest common divisor of the absolute values of three minor determinants of size $2{\times}2$ of the matrix of size $3{\times}2$ defined by the coefficients of $\ell_{\si_1},\ell_{\si_2}$ (which can be identified with the exterior product $\ell_{\si_1}{\wedge}\,\ell_{\si_2}$).
\sk
If $d_{\si}\eq0$ and the number of 2-dimensional faces containing $\si$ is three, set
\begin{equation*}
J_{\si}(s):={\rm mult}(\ell_{\si_1},\ell_{\si_2},\ell_{\si_3})\big/\bl((N_{\si_1}s\pl\nu_{\si_1})(N_{\si_2}s\pl\nu_{\si_2})(N_{\si_3}s\pl\nu_{\si_3})\br).
\end{equation*}
Here $\si_1,\si_2,\si_3$ are the three 2-dimensional faces containing $\si$ and ${\rm mult}(\ell_{\si_1},\ell_{\si_2},\ell_{\si_3})$ is the absolute value of the matrix defined by the coefficients of $\ell_{\si_1},\ell_{\si_2},\ell_{\si_3}$.
\sk
In general, let $\si_i$ for $i\in[1,r]$ be the 2-dimensional faces of $\Gp(f)$ containing $\si$ and such that the $\si_i\cap\si_{i+1}$ for $i\in[1,r]$ are 1-dimensional faces of $\Gp(f)$ with $\si_{r+1}\defs\si_r$. Then
\begin{equation*}
J_{\si}(s):=\msum_{i=3}^r\,{\rm mult}(\ell_{\si_1},\ell_{\si_{i-1}},\ell_{\si_i})\big/\bl((N_{\si_1}s\pl\nu_{\si_1})(N_{\si_{i-1}}s\pl\nu_{\si_{i-1}})(N_{\si_i}s\pl\nu_{\si_i})\br).
\end{equation*}
This can be defined more generally by choosing a simplicial partition of the cone spanned by $\ell_{\si_1},\dots,\ell_{\si_r}$. It is quite amazing that the sum is independent of the choice of a partition, see \cite[Lemma 5.1.1]{DL}. (This can be verified by using a computer in the situation of the proof of Theorem~\ref{T2}, see Remark~\ref{R3.2} below.) 

\section{Proof of Theorem~\ref{T2}} \label{S3}

\nin
Let $P,Q$ be the vertices of the unique compact face $\si_0$ of $\Gp(f)$, which is 1-dimensional. Taking a monomial $x^iy^jz^k$ of highest degree which divides $f$, we see that their coordinates are given by
\begin{equation*}
P=(i\pl ar,j\pl br,k)\q Q=(i,j,k\pl cr),
\end{equation*} 
where $a,b,c,r,i,j,k$ are non-negative integers such that $a\gess1$, $a\gess b$, $r\gess1$, $c\eq a\pl b$, $a,b$ are coprime in the case $b\gess1$, and $i,j,k$ are 0 or 1 (since $C$ is reduced). If $b\eq0$, we have $j\eq1$ (since $C$ is not a cone), and $f$ is a product of $z$ with a homogeneous polynomial $g$ of $x,y$ with degree $d{-}1$. The minimal spectral number of $g$ is $\tfrac{2}{d-1}$, which is different from $\tfrac{3}{d}$ unless $d\eq3$, where $f\eq xyz$ changing coordinates. The assertion is then easily verified using Remark~\ref{R3.1} below. So we assume $b\gess1$.
\sk
The singular points of $C$ contain $p_1\defs[0\col1\col0]$ always (since $a\sgt1$), $p_2\defs[1\col0\col0]$ if $b\nes1$ or $j\nes0$ or $k\nes0$ or $r\nes1$, and $p_3\defs[0\col0\col1]$ if $i\eq j\eq1$. The irreducible components of $C$ are rational curves passing through two of the three points $p_1,p_2,p_3$, and the number of irreducible components is $r\pl i\pl j\pl k$. This may be verified by considering the factorization of a polynomial of one variable of degree $r$ whose coefficients are those of $f$. Since this polynomial must be reduced (otherwise $f$ is non-reduced), its roots are mutually distinct and $f$ is Newton-nondegenerate.
\sk
There are five 2-dimensional faces $\si_i$ of $\Gp(f)$ with $i\ins[1,5]$. The coefficients of their corresponding linear functions $\ell^{(i)}$ for $i\ins[1,5]$ are respectively
\begin{equation*}
(1,0,0),\q(0,1,0),\q(0,0,1),\q(c,0,a),\q(0,c,b),
\end{equation*} 
and the corresponding $(N_{\si_i},\nu_{\si_i})$, which will be denoted by $(N_i,\nu_i)$, are respectively
\begin{equation*}
\begin{aligned}
&(i,1),\q(j,1),\q(k,1),\q(m,a\pl c),\q(n,b\pl c),\\
&\h{with}\q\q m\defs ci\pl ak\pl acr,\q n\defs cj\pl bk\pl bcr.
\end{aligned}
\end{equation*} 
We then see that $J_{\si_0}(s)$, $J_P(s)$, $J_Q(s)$ are given by
\begin{equation*}
\begin{aligned}
J_{\si_0}(s)&=c\big/\bl((ms\pl a\pl c)(ns\pl b\pl c)\br),\\
J_P(s)&=c^2\big/\bl((ks\pl 1)(ns\pl b\pl c)(ms\pl a\pl c)\br),\\
J_Q(s)&=b\big/\bl((is\pl 1)(js\pl 1)(ns\pl b\pl c)\br)\\
&\q+ac\big/\bl((ns\pl b\pl c)(ms\pl a\pl c)(is\pl 1)\br).
\end{aligned}
\end{equation*} 
Indeed, for the first two equalities it is enough to calculate the minor determinants of $\bl(\begin{smallmatrix}c&0&a\\0&c&b\end{smallmatrix}\br)$ and the determinant of $\bl(\begin{smallmatrix}0&0&1\\c&0&a\\0&c&b\end{smallmatrix}\br)$.
For the last one, we verify that the 2-dimensional faces containing $Q$ are $\si_1$, $\si_2$, $\si_5$, $\si_4$, and compute the determinants of $\bl(\begin{smallmatrix}1&0&0\\0&1&0\\0&c&b\end{smallmatrix}\br)$ and $\bl(\begin{smallmatrix}0&c&b\\c&0&a\\1&0&0\end{smallmatrix}\br)$. (Recall that there is another expression, but these coincide according to a computer algebra program, see Remark~\ref{R3.2} below.)
\sk
It remains to show that the pole at $s\eq{-}\tfrac{a+c}{m}$ and also the one at $s\eq{-}\tfrac{b+c}{n}$ in the case $p_2\ins{\rm Sing}\,C$ are poles of the topological zeta functions of singular points of $C$. This follows from the formula \eqref{2.1} for the curve case if $p_2\ins{\rm Sing}\,C$, since $\ell_4$, $\ell_5$ are pullbacks of the linear functions for the image of $\si_0$ by projections to $\R^2$. In the case $p_2\,{\notin}\,{\rm Sing}\,C$, where $b\eq r\eq1$, $j\eq k\eq0$, it seems quite nontrivial to show that $s\eq{-}\tfrac{b+c}{n}$ is a ``fake" pole of $Z^{\rm top}_{f,0}(s)$ associated with a $B_1$-facet, and we need \cite[Proposition 14]{LV} in an essential way. (One can verify it by using a computer, see Remark~\ref{R3.2} below.) This finishes the proof of Theorem~\ref{T2}.

\begin{rem}\label{R3.1}
In the case $f$ is a homogeneous polynomial, an embedded resolution can be obtained by the blowup of the origin followed by the base change of an embedded resolution of the projective hypersurface by the projection to the exceptional divisor of the first blowup. In view of the definition of local topological zeta function, this implies that $-\tfrac{3}{d}$ is the only possible pole of $Z^{\rm top}_{f,0}(s)$ which is not a pole of $Z^{\rm top}_{h,0}(s)$ for a local defining function $h$ of $C$.
\end{rem}

\begin{rem}\label{R3.2}
The above argument can be used to prove the formula \eqref{2.1} for our case by comparing it with the one for the two variable case. What is very amazing is that the {\it numerator\one} of the (modified) ``global" topological zeta functions of $C$ (which contains the contribution from $\PP^2\stm C$) is {\it divisible\one} by $s+\tfrac{3}{d}$ so that the pole $-\tfrac{3}{d}$ {\it disappears.} Note that we need the division by $ds\pl3$ for the passage from the (modified) ``global" zeta functions of $C$ to $Z^{\rm top}_{f,0}(s)$ (via the blowup at 0) by definition. Using Macaulay2 for instance, this can be seen as follows.
\ms
\vbox{\footnotesize\sf
\pv@S=QQ[a,b,i,j,k,r,s]; c=a+b; m=c*i+c*a*r+a*k; n=c*j+c*b*r+b*k; d=i+j+k+c*r;@
\pv@Z = -s/(s+1) *c*r/(m*s+a+c)/(n*s+b+c) + b/(i*s+1)/(j*s+1)/(n*s+b+c) +@
\pv@a*c/(n*s+b+c)/(m*s+a+c)/(i*s+1) + c^2/(k*s+1)/(n*s+b+c)/(m*s+a+c);@
\pv@Z1 = -s/(s+1)*r/(m*s+a+c) + c/(m*s+a+c)/(k*s+1) + a/(m*s+a+c)/(i*s+1);@
\pv@Z2 = -s/(s+1)*r/(n*s+b+c) + c/(n*s+b+c)/(k*s+1) + b/(n*s+b+c)/(j*s+1);@
\pv@R = 1/(i*s+1)/(j*s+1); Z - (Z1+Z2+R)/(d*s+3)@}
\ms
Here {\small\sf\verb@Z, Z1, Z2@} are the local topological zeta functions of $f$ and the two singular points of $C$, and {\small\sf\verb@R@} is the ``remaining" part of the (modified) ``global" zeta function of $C$. Notice that {\small\sf\verb@R@} is independent of $k\ins\{0,1\}$, since $\{z\eq0\}\stm{\rm Sing}\,C\cong\C^*$. In the cases $(i,j)\eq(0,0)$, $(1,0)$, and $(1,1)$, {\small\sf\verb@R@} is the contribution from $\PP^2\stm C$, $\,\{x\eq0\}\stm{\rm Sing}\,C\,({\cong}\,\C)$, and $[0\col0\col1]$ respectively. This computation shows that {\small\sf\verb@Z1+Z2+R@} is divisible by {\small\sf\verb@d*s+3@}, and the quotient coincides with {\small\sf\verb@Z@} in the field of fractions of the polynomial ring. This gives also a ``proof" of \eqref{2.1} for our case assuming it in the two variable case (and using a computer algebra program).
\sk
It is not recommended to remove ; at the end of the definition of ``{\small\sf\verb@Z@}''. If one prefers, ``{\small\sf\verb@S=QQ[a,b,i,j,k,r,s];@}'' may be replaced by ``{\small\sf\verb@S=QQ[s]; a=2; b=3; i=0; j=0; k=0; r=1;@}'' for instance.
\sk
In the case $p_2$ is not a singular point of $C$ (where $b\eq r\eq1$, $j\eq k\eq 0$), we see that $-\tfrac{b+c}{n}\,\bl({=}\,{-}\tfrac{a+2}{a+1}\br)$ is actually a ``fake" pole, and {\small\sf\verb@Z2@} coincides with {\small\sf\verb@1/(s+1)@}. This can be seen by setting ``{\small\sf\verb@S=QQ[a,i,s]; b=1; j=0; k=0; r=1;@}'' at the beginning. Note also that $s\eq{-1}$ is a simple pole, putting ``{\small\sf\verb@i=1;@}".
\sk
The coincidence with another expression of $J_Q(s)$ can be seen by the following:
\ms
\vbox{\footnotesize\sf
\pv@b/(i*s+1)/(j*s+1)/(n*s+b+c) + a*c/(n*s+b+c)/(m*s+a+c)/(i*s+1) -@
\pv@(b*c/(j*s+1)/(n*s+b+c)/(m*s+a+c) + a/(m*s+a+c)/(i*s+1)/(j*s+1))@}
\ms
As for the relation to the trace of the logarithmic vector field in \cite{BV}, it may be seen by comparing the following:
\ms
\vbox{\footnotesize\sf
\pv@u = -(3*m - d*(a+c))@
\pv@v = 3*n - d*(b+c)@
\pv@p = -a*(d-3*j) + b*(d-3*i)@
\pv@q = (a*r*j-b*r*i + c*r*j+b*r*k+b*c*r^2 - (c*r*i+a*r*k+a*c*r^2))/r@}
\msn
All of these turn out to be the same polynomial. The last one is the sum of the coefficients of a linear function $\ell$ vanishing on $P$ and $Q$, where $\ell$ is identified with a vector field $\xi$ of degree 0 annihilating $f$.
\end{rem}

\section{Proof of Theorem~\ref{T4}} \label{S4}

We first show the following (which are needed for the proofs of Theorems~\ref{T3} and \ref{T4}).

\begin{prop}\label{P4.1}
Let $Z$ be a reduced projective hypersurface of dimension at least $2$. Assume the hypersurface has only homogeneous isolated singularities and there is a nonzero linear function $\ell(\nu_1,\dots,\nu_n)\eq\msum_{i=1}^n\,c_i\nu_i$ on $\R^n$ with rational coefficients such that ${\rm Supp}_{\bf x}\one f\sst\ell^{-1}(0)$ with $f$ a defining polynomial of $Z$. Then $Z$ is a cone or smooth with $d\defs\deg f\eq2$.
\end{prop}

\begin{proof}
We consider the case $d\gess4$ first. Since $f$ is a homogeneous polynomial of degree $d$, we have ${\rm Supp}_{\bf x}\one f\sst\ell'\one{}^{-1}(1)$ with $\ell'(\nu_1,\dots,\nu_n)\defs\msum_{i=1}^n\,\tfrac{1}{d}\one\nu_i$. For $i\ins[1,n]$ with $c_i\ne 0$, we get a linear function $\ell^{(i)}$ with rational coefficients on $\R^{n-1}$ such that $p_i({\rm Supp}_{\bf x}\one f)$ is contained in $(\ell^{(i)})^{-1}(1)$ by elimination, setting $\ell^{(i)}\defs\ell'\mi \tfrac{1}{dc_i}\ell$. Here $p_i\col\R^n\tos\R^{n-1}$ is the $i\one$th projection (omitting the $i\one$th component). We may assume $\max_jc_j\gess{-}\min_jc_j$, replacing the $c_j$ with $-c_j$ if necessary.
If $c_i\eq\max_j|c_j|$, then the coefficients $w^{(i)}_j$ of $\ell^{(i)}$ are all non-negative, and $w^{(i)}_j\less\tfrac{2}{d}\less\tfrac{1}{2}$ (since $d\gess4$), hence the polynomial $f^{(i)}\defs f|_{x_i=1}$ has a weighted homogeneous isolated singularity at 0 with weights $w^{(i)}_j$. These $w^{(i)}_j$ must be strictly positive for any $j\nes i$, since we get non-isolated singularities otherwise. The Milnor number $\mu$ is equal to 1 if and only if $w^{(i)}_j\eq\tfrac{1}{2}$ for any $j\nes i$. Note that the Milnor number and the spectral numbers $\al_1,\dots\al_{\mu}$ (which are indexed weakly increasingly) of a weighted homogeneous polynomial are determined by the weights, see \cite{St} (and \cite[Section 7]{wh}, \cite[Section 1.5]{JKSY}):
\begin{equation*}
\msum_{k=1}^{\mu}\,t^{\al_k}=\mprod_{j\ne i}\,(t\mi t^{w_j})/(t^{w_j}\mi1),\q\mu=\mprod_{j\ne i}\,(w_j^{-1}\mi1),
\end{equation*}
where the $w^{(i)}_j$ are denoted by $w_j$. We then see that the weights $w^{(i)}_j$ must be constant, since the weights in the homogeneous case with $\mu\nes1$ can be read from the $\al_k\mi\al_1$ for $k\ins[2,n]$.
\sk
We thus proved that the $c_j$ are independent of $j\nes i$. This implies the inclusion
\begin{equation*}
{\rm Supp}_{\bf x}\one f\subset\bl\{\nu\ins\N^n\mid\msum_j\,\nu_j\eq d,\,\nu_i\eq c\one\msum_{j\nes i}\,\nu_j\br\}
\end{equation*} 
for some $c\gess0$. Then $f$ is either divisible by $x_i$ or a polynomial of $x_j$ for $j\nes i$. Here the first case does not occur, since $Z$ has non-isolated singularities otherwise. So the assertion follows for $d\gess4$.
\sk
It remains to show the case $d\eq3$. (Note that the case $d\eq2$ follows from the classification theory of quadratic forms over $\C$, since these are classified only by the rank.)
We assume $Z$ is not a cone. Then all the singularities of $Z$ must be ordinary double points, since $d\eq3$. We may assume $c_1\eq\max_i|c_i|\eq1$ changing the order of variables and multiplying $\ell$ by a rational number. Then $f$ cannot contain a monomial of the form $x_1^2x_i$ for any $i\ins[1,n]$ (that is, the coefficient of this monomial in $f$ vanishes), since $2\pl c_i\eq0$. Hence $f^{(1)}$ has an singularity at the origin substituting $x_1\eq1$. Moreover we have
\begin{equation*}
c_i\less0\q\h{for any}\,\,\,i\ins[2,n]. 
\end{equation*}
Indeed, $f$ must contain a monomial of the form $x_1x_ix_k$ for some $k\ins[1,n]$ in order that the singularity of $f^{(1)}$ is an ordinary double point after substituting $x_1\eq1$. (Indeed, monomials of degree 3 can be neglected for the judgement of whether $\mu\eq1$ or not, since it is decided by the determinant of the matrix whose $(i,j)$-component is the value at 0 of $\dd_{x_i}\dd_{x_j}f^{(1)}$ for $i,j\ins[2,n]$.) This is however impossible if $c_i\sgt0$ (where $1\pl c_i\pl c_k\eq0$ and $|c_k|\less1$).
\sk
Since $f$ is irreducible, $f$ must contain a monomial of the form $x_jx_{j'}x_{j''}$ with $j,j',j''\ins[2,n]$ (otherwise $f$ is divisible by $x_1$). Here $c_j,c_{j'},c_{j''}$ must be 0, since they are nonpositive with $c_j\pl c_{j'}\pl c_{j''}\eq0$. In order that $f^{(1)}$ has an ordinary double point at 0, the polynomial $f$ must contain a monomial of the form $x_1x_jx_k$ for some $k\ins[1,n]$ as above. Here $c_k$ must be $-1$, since $c_j\eq0$ this time. This $k$ is unique and $c_i\eq0$ for any $i\nes1,k$, applying the above argument to $-\ell$. We may assume $k\eq 2$ changing the order of variables. We then see that $f\eq x_1x_2\one g\pl h$ with $g,h$ homogeneous polynomials of $x_i$ for $i\ins[3,n]$ with degree 1 and 3 respectively. However this cannot give an ordinary double point after substituting $x_1\eq1$ or $x_2\eq1$ (since $n\gess4$), where the homogeneous polynomial $h$ of degree 3 can be neglected for the judgement of whether $\mu\eq1$ or not. So the assertion follows also for $d\eq3$. This finishes the proof of Proposition~\ref{P4.1}.
\end{proof}

\begin{rem} \label{R4.2}
Let $b_f(s),b_{f^m}(s)$ be the Bernstein-Sato polynomials of $f$ and $f^m$ with $m\ins\Z_{>0}$. These are polynomials of lowest degrees (or generators of ideals) satisfying the equations
\begin{equation*}
b_f(s)f^s=P(s)f^{s+1},\q b_{f^m}(s)(f^m)^s=P^{(m)}(s)(f^m)^{s+1},
\end{equation*}
with $P(s),P^{(m)}(s)\ins\D_{\C^n,0}[s]$. The first equality implies for $k\ins\N$ that
\begin{equation*}
b_f(ms\pl k)f^{ms+k}=P(ms\pl k)f^{ms+k+1}.
\end{equation*}
We then get the divisibility
\begin{equation*}
\mprod_{k=0}^{m-1} b_f(ms\pl k)\in b_{f^m}(s)\one\Q[s].
\end{equation*}
On the other hand we see that
\begin{equation*}
b_{f^m}\bl(\tfrac{s}{m}\br)\in b_f(s)\one\Q[s].
\end{equation*}
These show for instance that the maximal root of $b_{f^m}(s)$ coincides with that of $b_f(s)$ divided by $m$. (This would be quite well known to specialists, see also \cite[Example 4.11]{Bu}.)
\end{rem}

\begin{rem} \label{R4.3}
For $j\ins[1,n]$, let $\eta_j$ be the exterior product of the $\ddd x_i$ for $i\nes j$ up to sign so that $\ddd x_i{\wedge}\one\eta_j\eq\de_{i,j}\one\om$ with $\om\defs\ddd x_1{\wedge}\cdots{\wedge}\one\ddd x_n$. The exterior product of $\ddd f$ with an $(n{-}1)$-form $\msum_{i,j}\,c_{i,j}x_i\eta_j$ for $c_{i,j}\ins\C$ (which induces the $E_0$-differential of the pole order spectral sequence) is equal to $\xi(f)\om$ with $\xi\defs\msum_{i,j}\,c_{i,j}x_i\dd_{x_j}$. The image of the $(n{-}1)$-form by the exterior derivative $\ddd$ (which induces the $E_1$-differential of the spectral sequence) is given by $\om$ multiplied by the trace $\msum_i\,c_{i,i}$, since $\dd_{x_j}(c_{i,j}x_i)\eq\de_{i,j}c_{i,j}$.
(The first relation is quite implicit in the calculations in \cite[Section 8]{wh}. The second seems to have been forgotten to consider while Propositions 8.1--2 were studied in that paper.)
\end{rem}

\begin{rem}\label{R4.4}
Let $h$ be a homogeneous polynomial of $n{-}1$ variables with degree $e$ having an isolated singularity at 0. Then the poles of the topological zeta function of $h$ are contained in the set $\bl\{-1,-\tfrac{n-1}{e}\br\}$. In the case $\tfrac{n-1}{e}\eq\tfrac{n}{d}$, we have $(d\mi e)n\eq d$, hence $d\gess n\gess4$ and $e\gess3$. Indeed, $d\mi e\eq\tfrac{d}{n}\less\tfrac{d}{4}$, and then $e\gess\tfrac{3d}{4}\gess3$. Note also that $\tfrac{n}{d}\eq\tfrac{n-1}{e}\eq\tfrac{1}{a}$ with $a\eq d\mi e$, and $\tfrac{e}{d}\eq\tfrac{n-1}{n}\gess\tfrac{3}{4}$ if $n\gess4$.
\end{rem}

\begin{rem}\label{R4.5}
If $f\eq\msum_{|\nu|=d'}\,a_{\nu}{\bf x}^{\nu}$ is a homogeneous polynomial of degree $d'$ having an isolated singularity at 0, then for any $i\ins[1,n]$, there is $j\ins[1,n]$ with $a_{\nu}\nes0$ for $\nu\eq (d'{-}1){\bf e}^{(i)}\pl{\bf e}^{(j)}$, where ${\bf e}^{(i)}\eq(e^{(i)}_1,\dots,e^{(i)}_n)\ins\N^n$ with $e^{(i)}_k\eq\de_{i,k}$. (Indeed, ${\bf x}^{(d'-1){\bf e}^{(i)}}$ must be contained in the Jacobian ideal $(\dd f)$ generated by the derivatives $\dd_{x_i}f$ for $i\ins[1,n]$.)
\end{rem}

\begin{rem}\label{R4.6}
Let $Z$ be a projective hypersurface defined by a homogeneous polynomial $f$ which is annihilated by a nonzero ``nilpotent" vector field $\xi$ of degree 0. By Jordan normal form we may assume $\xi\eq\msum_{i\in I}\,x_i\dd_{x_{i+1}}$ with $I\eq\{1,\dots,n\}\stm\{n'_1,\dots,n'_r\}$, where $n'_k\eq\msum_{j=1}^k\,n_j$ with $\{n_j\}$ a weakly increasing sequence of strictly positive integers for $j,k\ins[1,r]$ with $r\ins\Z_{>0}$ and $n'_r\eq n$. We calculate the minor determinants of size $2{\times}2$ of the matrix consisting of the coefficients of $\xi$ and the Euler vector field $\xi'\defs\msum_{i=1}^n\,x_i\dd_{x_i}$, which are logarithmic vector fields along $\{f\eq0\}$. These minor determinants must vanish on any {\it isolated\one} singular point of $\{f\eq0\}\sst\C^n$. Indeed, the vector field $\xi$ induces a nonvanishing vector field along $Z$ on the open subset of $\PP^{n-1}$ where some of the minor determinants is nonzero, since $[\xi',\xi]\eq0$, that is, the Lie derivative of $\xi$ by $\xi'$ vanishes. (One may use also the Frobenius theorem on the integrability of ``distributions".) We then see that any isolated singular point of $Z$ is contained in $\mcap_{i\in I}\{x_i\eq0\}$ by an increasing induction.
\end{rem}

\begin{rem}\label{R4.7}
We can extend \cite[Theorem 3.8]{bha} to the case $k\eq d$, where $m_1\eq3$ and $q\eq0$ in \cite[(3.8.1)]{bha}. Indeed, we have the inclusions
\begin{equation*}
F_{-n}\B_f\,({=}\,\OO_Y)\subset\D_Y[s]f^s\subset V^{>0}\B_f,
\end{equation*}
where $f^s$ is identified with $\de(f{-}t)$. These imply that
\begin{equation*}
F_{1-n}\Gr_V^1\B_f\subset\Gr_V^1\D_X[s]f^s=:G_0\Gr_V^1\B_f\subset\Gr_V^1\B_f,
\end{equation*}
where the Hodge filtration $F$ is shifted by 1 taking $\Gr_V^1$ (but not for $\Gr_V^0$), and $F$ is indexed as in the right $\D$-module case, see \cite{mhp}. Moreover, by \cite[(3.5.7)]{bha}, we have the isomorphism of filtered $\D_Y$-modules
\begin{equation*}
\Gr^W_{n+1}(\Gr_V^0\B_f,F)=(i_0)^{\D}_*\bl(\Gr^W_{n+1}H^{n-1}(F_{\!f,0},\C)_1,F\br).
\end{equation*}
Here $(i_0)^{\D}_*$ denotes the direct image as a filtered $\D$-module, and
$\Gr_V^0\B_f$ (whose induced filtration $F$ is unshifted) is the filtered $\D_Y$-module underlying $\varphi_{f,1}\R_{h,Y}[n{-}1]$, and is identified with $\Gr_V^1\B_f/{\rm Ker}\,N$ by $\dd_t$ (which shifts the filtration $F$ by 1). The weight filtration $W$ on $\Gr_V^0\B_f$ is the monodromy filtration shifted by $n$ so that $\Gr^W_{n+j}\Gr_V^0\B_f\eq0$ if $|j|\sgt1$. The assumed inequality then implies the non-vanishing of $F_{1-n}\Gr^W_{n+1}\Gr_V^0\B_f$ using the short exact sequence \cite[(3.6.1)]{bha}.
\sk
Since $G_{-1}\eq0$ in the notation of \cite[Section 1.1]{bha}, these are sufficient for the proof of the generalization of \cite[Theorem 3.8]{bha} to the case $k\eq d$. Indeed, the order of nilpotence of $N$ increases by 1, passing from $\Gr_V^0\B_f\eq\Gr_V^1\B_f/{\rm Ker}\,N$ to $\Gr_V^1\B_f$ by using the hard Lefschetz property and the coincidence of the $N$-primitive parts of the $W$-graded quotients of the unipotent monodromy parts of the nearby and vanishing cycle Hodge modules $\psi_{f,1}\R_Y[n{-}1]$, $\varphi_{f,1}\R_Y[n{-}1]$ (see for instance \cite[(2.2.5)]{KLS}, \cite[Remark 2.3a]{FPS}).
\end{rem}

The following will be needed in the proof of Theorem~\ref{T4}.

\begin{lem}\label{L4.8}
Assume $Z$ is reduced and has only homogeneous isolated singularities with $n\sgt3$. Assume there is a singular point of $Z$ defined analytic-locally by a homogeneous polynomial of degree $e\slt d$ with $\tfrac{n-1}{e}\eq\tfrac{n}{d}$. Then there is no nonzero ``nilpotent" vector field $\xi$ of degree $0$ annihilating $f$.
\end{lem}

\begin{proof}
Assume there is a nonzero ``nilpotent" vector field $\xi$ of degree 0 annihilating $f$. In \cite[Theorem 2.4]{dPW2}, hypersurfaces of $\PP^{n-1}$ having only isolated singularities and whose defining polynomial $f$ is annihilated by a ``nilpotent" vector field $\xi$ are classified. We have only four Cases: 2 ($d\gess3$, $n\gess3$), 3 ($d\eq4$, $n\gess4$), 4 ($d\eq3$, $n\gess5$), and 21 ($d\eq3$, $n\gess5$), depending on the types of Jordan blocks (in the notation of that paper). By Remark~\ref{R4.4}, Cases 4 and 21 are excluded (since $n\sgt d$), and in Case 3, the assertion follows from \cite[Lemma 6.6]{dPW2}, where it is shown that the singularities are semi-weighted-homogeneous with {\it non-constant\one} weights.
\sk
So it remains to consider Case 2, where $f$ is annihilated by $\xi\eq x\dd_y\pl y\dd_z$ with $n\gess4$. Put $p\defs y^2\mi2xz$. By \cite{We} (see also \cite[Lemma 4.1]{dPW2} and \cite[Proposition 8.1]{wh}) we have
\begin{equation*}
f\eq\msum_{i=0}^{[d/2]}\one\msum_{k=0}^{d-2i}\,p^ix^kg_{i,k}\q\h{with}\q g_{i,k}\ins\C[v_4,\dots,v_n]^{d-2i-k},
\end{equation*}
where $x,y,z,v_4,\dots,v_n$ are coordinates of $\C^n$ and $\C[v_4,\dots,v_n]^m$ denotes the vector space consisting of homogeneous polynomials of degree $m$ (including 0) for $m\ins\N$. (This is easily reduced to the case $n\eq3$, since the polynomial ring is the tensor product of $\C[x,y,z]$ with $\C[v_4,\dots,v_n]$ over $\C$, and monomials are linearly independent in the latter polynomial ring.)
\sk
The singular locus of $Z\sst\PP^{n-1}$ is contained in $\{x\eq y\eq0\}$, see Remark~\ref{R4.6}. Assume that a homogeneous isolated singularity of local degree $e$ with $\tfrac{n-1}{e}\eq\tfrac{n}{d}$ is contained in $\{x\eq y\eq z\eq0\}$. Then we must have $d\eq n\eq4$ applying Remark~\ref{R4.5} to $z$. Indeed, $\tfrac{e}{d}\gess\tfrac{3}{4}$ by Remark~\ref{R4.4}, and we get $\nu_3\less\tfrac{d}{2}$ for any $\nu\ins\N^n$ with $a_{\nu}\nes0$ in view of the definition $p\defs y^2\mi2xz$ (since $z$ appears {\it only here\one}), where $f\eq\msum_{\nu}\,a_{\nu}{\bf v}^{\nu}$ with ${\bf v}\eq(v_1,\dots,v_n)$ and $v_1\eq x$, $v_2\eq y$, $v_3\eq z$. The only possible case here is the one with $d\eq n\eq4$ and $\tfrac{e}{d}\eq\tfrac{3}{4}$. However, we get a contradiction applying Remark~\ref{R4.5} to $z$ again in view of the definition of $p$ (that is, $p^2\eq 4\one x^2z^2\pl\cdots$).
\sk
The homogeneous singular point with $\tfrac{n-1}{e}\eq\tfrac{n}{d}$ is thus not contained in $\{z\eq0\}$. Let $[0\col0\col1\col c_4\col\cdots\col c_n]$ be its projective coordinates with $c_i\ins\C$ for $i\ins[4,n]$. Set
\begin{equation*}
h\defs f|_{z=1}\eq\msum_{i=0}^{[d/2]}\one\msum_{k=0}^{d-2i}\,p'{}^ix^kg_{i,k}\q\h{with}\q p'\defs y^2\mi2x.
\end{equation*}
Assume $(c_4,\dots,c_n)\eq0$. Applying Remark~\ref{R4.5} to $y$ and calculating powers of $p'$, we see that $g_{\frac{d}{2},0}\nes0$ with $d$ even, since the monomial $xy^{d-2}$ (which is contained in $p'{}^{d/2}$ and in $h$) is the only way to satisfy the conclusion of Remark~\ref{R4.5} for $y$ with $d'\eq e\slt d$, where the latter remark is applied to $h_e$ with $h\eq\msum_{k\ges e}\,h_k$ the decomposition of $h$ by homogeneous polynomials $h_k$ of degree $k$ (which corresponds to the decomposition of $f$ by the degree of $z$, that is, $f\eq\msum_{k\ges e}\,h_kz^{d-k}$). However, this is impossible, since $e\eq d{-}1$ and $x^2y^{d-4}$ is also contained in $p'{}^{d/2}$ and in $h$ (using $d\gess4$).
\sk
Assume finally $(c_4,\dots,c_n)\nes0$. We may assume $c_i\eq0$ for any $i\gess5$ changing $\C$-linearly the coordinates $v_4,\dots,v_n$. Set ${\bf v}'\defs(x,y,v_5,\dots,v_n)$. We have $g_{i,k}|_{{\bf v}'=0}\eq b_{i,k}\one v_4^{d-2i-k}$ with $b_{i,k}\ins\C$, and
\begin{equation*}
\begin{aligned}
\dd_xh|_{{\bf v}'=0}&=({-}2b_{1,0}\pl b_{0,1}v_4)v_4^{d-2},\\
\dd_yh|_{{\bf v}'=0}&=0,\\
\dd_{v_4}h|_{{\bf v}'=0}&=d\,b_{0,0}\one v_4^{d-1},\\
\dd_{v_i}h|_{{\bf v}'=0}&=b'_i\one v_4^{d-1}\q\h{for}\,\,\,i\ins[5,n],\\
\end{aligned}
\end{equation*}
where $b'_i\ins\C$ for $i\ins[5,n]$.
Since $h$ has an {\it isolated\one} singularity {\it not\one} at the origin, these imply that $b_{0,0}\eq0$, $b_{1,0}\one b_{0,1}\nes0$, $2b_{1,0}\eq b_{0,1}c_4$, and $b'_i\eq 0$ for $i\ins[5,n]$. We then get that
\begin{equation*}
(\dd_y\dd_yh)|_{{\bf v}'=0,\,v_4=c_4}=2b_{1,0}\one c_4^{d-2}\nes 0.
\end{equation*}
Considering the {\it Taylor expansion\one} of $h$ at $(0,0,c_4,0,\dots,0)$, we see that $e$ must be 2, but this contradicts Remark~\ref{R4.4}. So Lemma~\ref{L4.8} follows.
\end{proof}

We now give a proof of Theorem~\ref{T4}.

\subsection{Case I} Assume $n\eq3$.
We may assume $Z$ is reduced by using Remarks~\ref{R1} and \ref{R4.2}, and moreover $d\gess3$, since the case $d\eq2$ is trivial.
It is enough to compare the order and multiplicity of $-\tfrac{3}{d}$ as a pole and a root of $Z^{\rm top}_{f,0}(s)$ and $b_{f,0}(s)$ respectively in view of the argument in Remark~\ref{R3.1} together with the definition of local zeta functions.
\sk
If $f$ is not annihilated by any vector field of degree 0, then the assertion follows from the arguments in the proof of \cite[Theorem 3.8]{bha}. The argument is non-trivial in the case $Z$ has a singular point whose minimal spectral number coincides with $\tfrac{3}{d}$, since this number is a pole of the local topological zeta function of a local defining function of $Z$ and the one for $f$ may have a pole of order 3 if $d\eq3$ and 2 otherwise along $s\eq{-}\tfrac{3}{d}$. However the non-existence of a vector field of degree 0 annihilating $f$ implies the vanishing of $\Gr_P^1H^1(F_{\!f},\C)_{\la}$ with $\la\defs e^{-2\pi\sqrt{-1}\one 3/d}$, where $F_{\!f}$ is the Milnor fiber and the pole order filtration $P$ on $H^1(F_{\!f},\C)$ coincides with the Hodge filtration $F$, see \cite[Proposition 2.2]{DiSt}. So we can apply \cite[Theorem 3.8]{bha} in the case $d\sgt3$. For $d\eq3$, we can apply a similar argument, see Remark~\ref{R4.7}.
\sk
If $f$ is extremely degenerated, the computation in Remark~\ref{R3.2} implies that the order of pole at $s\eq{-}\tfrac{3}{d}$ is at most 2, and the assumption of \cite[Theorem 3.8]{bha} is satisfied. Indeed, if $f$ is annihilated by a nonzero diagonal vector field of degree 0 whose trace is 0, then the poles of the local topological zeta functions of the {\it two\one} singular points coincide, see Remark~\ref{R3.2}. In the case $C$ has only one singular point, its minimal exponent is $\tfrac{2a+1}{a(a+1)}$ and $\tfrac{2a+1}{(a+1)^2}$ with $d\eq a\pl1$ and $a\pl2$ for $i\eq0$ and $1$ respectively.
\sk
By Lemma~\ref{L1} the remaining case consists of polynomials in \cite[Remark 8.5]{wh}, where the assertion follows from the computation there, since the minimal spectral number of any singular point of $Z$ does not coincide with $\tfrac{3}{d}$. So Theorem~\ref{T4} is shown in the case $n\eq3$.

\subsection{Case II} Assume $n\gess4$.
We may assume $Z$ is reduced by using Remarks~\ref{R1} and \ref{R4.2}, and moreover $Z$ has a homogeneous isolated singularity of local degree $e$ with $\tfrac{n-1}{e}\eq\tfrac{n}{d}$, since otherwise the assertion is easily verified employing Remark~\ref{R3.1}.
\sk
In the case there is no nonzero vector field of degree 0 annihilating $f$, the graded eigenspace $\Gr_P^{n-2}H^{n-2}(F_{\!f},\C)_{\la}$ vanishes with $\la\defs e^{-2\pi\sqrt{-1}\one n/d}$ by the pole order spectral sequence using Remark~\ref{R4.3}. The assertion then follows from \cite[Theorem 3.8]{bha} if $d\sgt n$, and we apply its slight extension if $d\eq n$, see Remark~\ref{R4.7}.
\sk
Assume now there is a nonzero vector field of degree 0 annihilating $f$. Its ``semisimple part" must vanish by Proposition~\ref{P4.1} and Lemma~\ref{L1}. So this must be a nonzero ``nilpotent" vector field, however this contradicts Lemma~\ref{L4.8}. 
This finishes the proof of Theorem~\ref{T4}.

\section{Proof of Lemma~\ref{L1}} \label{S5}

Lastly we present an entirely elementary proof of Lemma~\ref{L1} generalizing an argument in the proof of \cite[Proposition 8.1]{wh}, see also \cite[Lemma 3.11]{BV}.
\sk
Let $\xi\eq\msum_{i,j}\,c_{i,j}x_i\dd_{x_j}$ be a vector field of degree 0 annihilating a homogeneous polynomial $f\eq\msum_{\nu\in\N^n}\,a_{\nu}{\bf x}^{\nu}$ of degree $d$. We may assume the matrix $A\eq(c_{i,j})$ is in Jordan normal form, see \cite[Remark 8.4]{wh}. Set $\xi'\defs\msum_i\,c_{i,i}x_i\dd_{x_i}$, $\xi''\defs\xi\mi\xi'$. Note that $\xi''$ is a $\C$-linear combination of $x_j\dd_{x_{j+1}}$ for some $j\ins[1,n]$ with $\{j,j{+}1\}\sst I(\al')$ and $\al'\ins\C$, since $A$ is in Jordan normal form.
\sk
Put $\ell_{\xi'}(\nu)\defs\msum_i\,c_{i,i}\nu_i$ for $\nu\eq(\nu_1,\dots,\nu_n)\ins\N^n$. Set $I(\al)\defs\{i\,|\,c_{i,i}\eq\al\}$ for $\al\ins\C$. Put $\nu(\al)_i\eq\nu_i$ if $i\ins I(\al)$ and 0 otherwise. Then ${\bf x}^{\nu}\eq\mprod_{\al\in\C}\,{\bf x}^{\nu(\al)}$, that is, $\nu\eq\msum_{\al\in\C}\,\nu(\al)$. This is the decomposition by eigenvalues of $A$. We have
\begin{equation*}\ell_{\xi'}(\nu)\eq\msum_{\al\in\C}\,\al|\nu(\al)|\q\h{with}\q|\nu(\al)|\eq\msum_i\,\nu(\al)_i.
\end{equation*}
\sk
Set $\Xi_{\xi'}\defs\ell_{\xi'}^{-1}(0)\cap\Z^n$. We have to show that ${\rm Supp}_{\bf x}f\sst\Xi_{\xi'}$, that is $a_{\nu}\eq0$ if $\nu\,{\notin}\,\Xi_{\xi'}$. This can be shown by decreasing induction taking the graded pieces of an increasing  filtration $G$ on the vector space $\C[{\bf x}]^d$ (consisting of homogeneous polynomials of degree $d$) such that the action of the nilpotent part $\xi''$ can be neglected. This filtration can be defined so that $G_p$ is generated over $\C$ by monomials ${\bf x}^{\nu}$ with $\phi(\nu)\less p$ and $|\nu|\eq d$, where $\phi(\nu)$ is given for instance by $\msum_{i=1}^n\,(d{+}1)^i(\nu_i{+}1)$ on $\{\nu\ins\N^n\,|\,|\nu|\eq d\}$. Note that
\begin{equation*}
\xi'{\bf x}^{\nu}\eq\ell_{\xi'}(\nu){\bf x}^{\nu}\q\h{for}\,\,\,\,\nu\ins\N^n,
\end{equation*}
and the monomials ${\bf x}^{\nu}$ are $\C$-linearly independent in the polynomial ring $\C[{\bf x}]$. (This is the key point in the proof.) Here we can verify that for $j\ins[1,n]$ with $\{j,j{+}1\}\sst I(\al')$ and $\al'\ins\C$, the subset $\Xi_{\xi'}\sst\Z^n$ is {\it stable\one} by the self-map $\rho^{(j)}$ of $\Z^n$ defined by
\begin{equation*}
\rho^{(j)}(\nu)_i\eq\begin{cases}\nu_i\pl1&\h{if}\,\,\,i\eq j,\\ \nu_i\mi1&\h{if}\,\,\,i\eq j\pl1,\\ \nu_i&\h{otherwise,}\end{cases}
\end{equation*}
corresponding to $x_j\dd_{x_{j+1}}$ (where $\phi\bl(\rho^{(j)}(\nu)\br)\slt\phi(\nu)$ for $\nu\ins\N^n$), since the $|\nu(\al)|$ do not change by the $\rho^{(j)}$. This implies that the influence of the nilpotent part $\xi''$ is restricted inside $\msum_{\nu\in\Xi_{\xi'}\cap\N^n}\C{\bf x}^{\nu}$, since $A$ is in Jordan normal form. So Lemma~\ref{L1} follows.

\end{document}